\documentclass[10pt, reqno]{amsart}
\usepackage[utf8]{inputenc}
\usepackage{amsmath,latexsym,amssymb,amsthm,enumerate,amsmath,amscd,amsfonts,stmaryrd}
\usepackage[mathscr]{eucal}
\usepackage{comment}
\usepackage{xcolor}
\usepackage{cleveref}
\usepackage{enumitem}
\usepackage{cases}
\usepackage{mathtools}
\usepackage{url}
\usepackage[figuresright]{rotating}
\usepackage[indentafter]{titlesec}
\usepackage{amsaddr}
\titleformat{name=\section}{}{\thetitle.}{0.8em}{\centering\scshape}
\titleformat{name=\subsection}[runin]{}{\thetitle.}{0.5em}{\bfseries}[.]
\titleformat{name=\subsubsection}[runin]{}{\thetitle.}{0.5em}{\itshape}[.]
\titleformat{name=\paragraph,numberless}[runin]{}{}{0em}{}[.]
\titlespacing{\paragraph}{0em}{0em}{0.5em}
\titleformat{name=\subparagraph,numberless}[runin]{}{}{0em}{}[.]
\titlespacing{\subparagraph}{0em}{0em}{0.5em}

\theoremstyle{plain}
\newtheorem{theorem}{Theorem}[section]
\newtheorem{proposition}[theorem]{Proposition}
\newtheorem{corollary}[theorem]{Corollary}
\newtheorem{lemma}[theorem]{Lemma}
\theoremstyle{definition}

\newtheorem{remark}[theorem]{Remark}

\numberwithin{table}{section}

\newtheorem{chunk}[theorem]{}

\newcommand{\normal}{\mathrel{\triangleleft}}

\newcommand{\Z}{\mathbb{Z}}

\newcommand{\Aut}{\mathrm{Aut}}
\newcommand{\Inn}{\mathrm{Inn}}
\newcommand{\Out}{\mathrm{Out}}

\newenvironment{arabiclist}{%
	\begin{enumerate}
	}{%
	\end{enumerate}}

\titleformat{\paragraph}
{\normalfont\normalsize\bfseries}{\theparagraph}{1em}{}
\titlespacing*{\paragraph}
{0pt}{3.25ex plus 1ex minus .2ex}{1.5ex plus .2ex}

\titleformat{\subparagraph}
{\normalfont\normalsize\bfseries}{\thesubparagraph}{1em}{}
\titlespacing*{\subparagraph}
{0pt}{3.25ex plus 1ex minus .2ex}{1.5ex plus .2ex}

\numberwithin{equation}{section}
\setlist[enumerate]{label=\arabic*.,font=\upshape, ref=(\arabic*)}

\begin{document}

\title{On the structure of a poly-$\Z$ group}
\author{Madeline Weinstein}
\address{Northeastern University \\ 360 Huntington Ave, Boston, MA 02115 \\ weinstein.ma@northeastern.edu}
\date{14 January 2021}  \keywords{}
\maketitle

\begin{abstract}
    In this paper we study a certain class of polycyclic groups. We outline a method for constructing a poly-$\Z$ group $G_n$ by describing a process for selecting maps that are used to extend $G_i$ to $G_{i+1}$ for $1 \leq i \leq n-1$ and describe the multiplicative structure and automorphism groups of some poly-$\Z$ groups up to $G_3$.
\end{abstract}

\section{Introduction}
A \textit{polycyclic group} is a group $G$ with a series of normal subgroups
\[
    \{1\} = G_0 \normal G_1 \normal \dots \normal G_n = G
\]
where for all $1 \leq i \leq n$, $G_{i-1} \normal G_i$ and $G_i / G_{i-1}$ is cyclic. Such a series $\{G_0, \dots, G_n\}$ is called a \textit{polycyclic series} of $G$ and $G$ is called an \textit{$n$-step} polycyclic group. Let $g_1, \dots , g_n$ be a sequence of elements in $G$ such that $\langle g_i G_{i-1} \rangle =  G_i / G_{i-1}$ for all $1 \leq i \leq n$. Then $G = \langle g_1, \dots, g_n \rangle$ where each $G_i = \langle g_i, G_{i-1} \rangle$ and $\{g_1, \dots , g_n\}$ is called a \textit{polycyclic generating sequence} of $G$. The \textit{relative order} of $g_i$ for $1 \leq i \leq n$ is given by $o_i = |G_i / G_{i-1}| = |g_i G_{i-1}|$.

An element of the group $G$ is called a \textit{word}, and a word written in the form $g_1^{e_1} \dots g_n^{e_n}$ where $e_i \in \Z$ such that $e_i < o_i$ for $1 \leq i \leq n$ is called a \textit{normal word}. Every polycyclic group has a presentation called a \textit{polycyclic presentation}, written in form
\begin{equation}
	\left \langle
		g_1, \dots, g_n \middle\vert \begin{array}{ll} g_jg_ig_j^{-1} = u_{i, j} & 1 \leq i < j \leq n \\
		g_j^{-1}g_ig_j = v_{i, j} & 1 \leq i < j \leq n \\
		g_i^{o_i} = w_i & 1 \leq i \leq n, o_i < \infty \end{array}
	\right \rangle \label{pcpres}
\end{equation}
where $u_{i, j}$ and $v_{i, j}$ are normal words in the generators $g_1, \dots, g_{j-1}$, and $w_i$ is a normal word in the generators $g_1, \dots, g_{i-1}$. In a \textit{consistent} polycyclic presentation, any element in $G$ can be written uniquely as a normal word by iteratively applying the relations defined in the presentation.

Every polycyclic group has a consistent polycyclic presentation \cite{heo}, so it follows that their word problem (determining whether or not two words in a group are equal) is solvable. The \textit{collection algorithm} provides one method for obtaining such a solution; the unique normal word for a given arbitrary word is computed by iteratively applying the power and conjugacy relations of the group to subwords of the given word. The \textit{Hirsch length} of a polycyclic group is the number of infinite quotient groups $G_i / G_{i-1}$. An $n$-step \textit{poly-$\Z$ group} is an $n$-step polycyclic group with Hirsch length $n$, where $G_i / G_{i-1} \cong \Z$ for all $1 \leq i \leq n$.

In symmetric key cryptography schemes, a number of group-theoretic key exchange protocols rely on conjugation in a given group to retain secrecy. As polycyclic groups are finitely presented and have solvable word, conjugacy, and isomorphism decision problems, they are suitable platform groups for cryptosystems provided that their conjugacy problem is sufficiently difficult \cite{ek}, \cite{gk}. In 2014 Cavallo and Kahrobaei proposed a polycyclic group to be used as a platform group for a secure key exchange protocol, relying on reducing the NP-complete twisted subset sum problem to the conjugacy problem in that group to prove that the group would be sufficiently secure \cite{ck1}. The authors later issued an erratum stating that the group's construction was flawed such that their main result was invalidated, as not all of the maps they used to extend $G_i$ to $G_{i+1}$ were automorphisms \cite{ck2}.

Initially, we began this project with the intent to propose a new polycyclic group such that a reduction exists from the twisted subset sum problem to its conjugacy problem. However, while we found references focusing on automorphism groups of finite cyclic groups (see \cite{bc}, \cite{bcm}, \cite{zl}), we found that descriptions of infinite polycyclic groups and their automorphism groups were largely missing from existing literature. The purpose of this paper is to study and describe poly-$\Z$ groups, a specific class of polycyclic groups. We begin by discussing a method of constructing poly-$\Z$ groups in \Cref{s_pzgroups} and in \Cref{s_12stepgroups}, we describe all $1$- and $2$- step poly-$\Z$ groups and discuss their multiplicative and automorphism group structure. Our main result is the set of descriptions of all $3$-step poly-$\Z$ groups up to isomorphism having a specific $2$-step poly-$\Z$ group as a base (Theorems \ref{b1_automorphisms}, \ref{a0_automorphisms}, \ref{a1_automorphisms}, \ref{b0_automorphisms}).

\section{Construction of poly-$\Z$ groups} \label{s_pzgroups}
Let $H$ be a group with identity $1_{H}$. For each $\varphi \in \Aut(H)$, we define $\overline{\varphi} \colon \Z \to \Aut(H)$ as a homomorphism such that $\overline{\varphi}(k) = \varphi^k$ for all $k \in \Z$. Set $G = H \rtimes_{\overline{\varphi}} \Z$. Then $G$ is a group with multiplication defined as
\[
    (h_1,k_1)(h_2,k_2) = (h_1 \varphi^{k_1}(h_2), k_1+k_2)
\]
for $h_1,h_2 \in H$ and $k_1,k_2 \in \Z$. If we identify $H$ with the subgroup $\{ (h, 0) \mid h \in H\}$ of $G$ where $h \in H$ corresponds to $(h, 0) \in G$ and set $g = (1_H, 1)$, then
\[
    G = \{ hg^k \mid h \in H, k \in \Z \} \quad \text{ and } \quad g^khg^{-k} = \varphi^k(h)
\]
for all $h \in H$ and $k \in \Z$. Then $H \normal G$ with $G/H = \langle gH \rangle \cong \Z$.

\Cref{inn_isomorph} is a particular case of a remark from \cite[\S 1.6]{cs} that gives a sufficient condition for two semidirect products to be isomorphic. The proof of the remark was not readily available, so we give a proof for the convenience of the reader using \Cref{A*}.
\begin{lemma}\label[lemma]{A*}
    Let $H$ be a group with $\alpha \in \Aut(H)$ and $a \in H$. Set 
    \[
        A_k \coloneqq
        \begin{cases}
            e &\text{ if } k = 0 \\
            a \alpha(a) \dots \alpha^{k-1}(a) &\text{ if } k>0 \\
            \alpha^{-1}(a^{-1}) \dots \alpha^k(a^{-1}) &\text{ if } k<0.
        \end{cases}
    \]
    Then for all $k, k_1, k_2 \in \Z$, the following properties hold:
    \begin{arabiclist}
        \item $A_k = \alpha^{k}(A_{-k}^{-1})$ \label{*.1} 
        \item $A_{k_1}^{-1}A_{k_1+k_2} = \alpha^{k_1}(A_{k_2})$. \label{*.2}
    \end{arabiclist}
\end{lemma}
    
\begin{proof}
    By definition of $A_k$, it is immediately clear that \ref{*.1} holds for $k=0$. Now observe that for $k>0$,
    \[
        A_k
        = a \alpha(a) \dots \alpha^{k-1}(a)
        = \alpha^k(\alpha^{-k}(a) \dots \alpha^{-1}(a))
        = \alpha^k(A_{-k}^{-1})
    \]
    and for $k<0$,
    \[
        A_k
        = \alpha^{-1}(a^{-1}) \dots \alpha^k(a^{-1})
        = \alpha^k(\alpha^{-k-1}(a^{-1}) \dots \alpha(a^{-1}) a^{-1})
        = \alpha^k(A_{-k}^{-1}).
    \]
        
    Now, consider all cases of the signs of integers $k_1$ and $k_2$. It is immediately clear that \ref{*.2} holds when either $k_1$ or $k_2$ is $0$. For any $k_1 \in \Z$, if $k_1 + k_2 = 0$ then $k_2 = -k_1$ and \ref{*.2} is true by \ref{*.1}. For $k_1, k_2 > 0$, we have
    \[
        A_{k_1}^{-1}A_{k_1+k_2}
        = \alpha^{k_1}(a) \dots \alpha^{k_1+k_2-1}(a)
        = \alpha^{k_1}(A_{k_2}).
    \]
    If $k_1 > 0$ and $k_2 < 0$, it remains to be shown that \ref{*.2} is true for $k_1 + k_2 > 0$ and $k_1 + k_2 < 0$. In either case,
    \[
        A_{k_1}^{-1}A_{k_1+k_2}
        = \alpha^{k_1-1}(a^{-1}) \dots \alpha^{k_1+k_2}(a^{-1})
        = \alpha^{k_1}(A_{k_2}).
    \]
    Similarly, if $k_1 < 0$ and $k_2 > 0$, then whether $k_1 + k_2$ is greater than or less than $0$,
    \[
        A_{k_1}^{-1}A_{k_1+k_2}
        = \alpha^{k_1}(a) \dots \alpha^{k_1+k_2-1}(a)
        = \alpha^{k_1}(A_{k_2}).
    \]
    Finally, for $k_1, k_2 < 0$ we have
    \[
        A_{k_1}^{-1}A_{k_1+k_2}
        = \alpha^{k_1-1}(a^{-1}) \dots \alpha^{k_1+k_2}(a^{-1})
        = \alpha^{k_1}(A_{k_2}).
    \]
\end{proof}

\begin{proposition}\label[proposition]{inn_isomorph}
    Let $H$ be a group and let $\alpha, \beta \in \Aut(H)$ such that there exists $\iota \in \Inn(H)$ such that $\beta = \iota \circ \alpha$. Then
    \[
        H \rtimes_{\overline{\beta}} \Z \cong H \rtimes_{\overline{\alpha}} \Z.
    \]
    In particular, $H \rtimes_{\overline{\iota}} \Z \cong H \times \Z$ for all $\iota \in \Inn(H)$.
\end{proposition}

\begin{proof}
    Since $\iota$ belongs to $\Inn(H) = \{\iota_b \mid \iota_b(h) = bhb^{-1} \text{ for some } b \in H \text{ and all } h \in H\}$, there exists some $a \in H$ such that $\iota = \iota_a$. Moreover, it can be shown that for all $k \in \Z$,
    \[
        \beta^k(h) = \iota_{A_k} \circ \alpha^k(h)
    \]
    where $A_k$ is defined as described in \Cref{A*}.

    We then define $\phi \colon H \rtimes_{\overline{\beta}} \Z \to H \rtimes_{\overline{\alpha}} \Z$ such that
    \[
        \phi(h,k) = (hA_k, k).
    \]
    It is clear that $\phi$ is a bijection with its inverse given by
    \[
        \phi^{-1}(h,k) = (hA_k^{-1}, k).
    \]
    It remains to show that $\phi$ is a homomorphism. We have that for any $h_1,h_2 \in H$ and any $k_1,k_2 \in \Z$,
    \[
        \phi(h_1,k_1)\phi(h_2,k_2)
        = (h_1A_{k_1},k_1)(h_2A_{k_2},k_2)
        = (h_1A_{k_1}\alpha^{k_1}(h_2A_{k_2}), k_1+k_2)
    \]
    and using \Cref{A*} \ref{*.2} in the last equality of the following, we see that
    \begin{align*}
        \phi((h_1,k_1)(h_2,k_2))
        = \phi(h_1 \beta^{k_1}(h_2), k_1+k_2)
        &= \phi(h_1A_{k_1}\alpha^{k_1}(h_2)A_{k_1}^{-1}, k_1+k_2) \\
        &= (h_1A_{k_1}\alpha^{k_1}(h_2)A_{k_1}^{-1}A_{k_1+k_2}, k_1+k_2) \\
        &= (h_1A_{k_1}\alpha^{k_1}(h_2A_{k_2}), k_1+k_2).
    \end{align*}
    Therefore $\phi(h_1,k_1)\phi(h_2,k_2) = \phi((h_1,k_1)(h_2,k_2))$, so $\phi$ is an isomorphism.
    
    To finish the proof, we notice that if $\alpha$ is the identity, then $\beta = \iota$ and $H \rtimes_{\overline{\iota}} \Z$ is isomorphic to $H \times \Z$.
\end{proof}

We go on to discuss another condition for isomorphic semidirect products, which generalizes a result from \cite{alm}.

\begin{proposition} \label[proposition]{conjugation}
    Let $G$ be a group such that $\overline{\alpha}, \overline{\beta}: \Z \to \Aut(G)$ are homomorphisms given by $\overline{\alpha}(k) =  \varphi^k$ and $\overline{\beta}(k) = \beta^k$ for some $\alpha, \beta \in \Aut(G)$. If $\alpha$ is conjugate to $\beta$ in $\Aut(G)$, then $G \rtimes_{\overline{\alpha}} \Z \cong G \rtimes_{\overline{\beta}} \Z$.
\end{proposition}

\begin{proof}
    Let $\beta = \psi \alpha \psi^{-1}$ for some $\psi \in \Aut(G)$. Then inductively, $\beta^k = \psi \alpha^k \psi^{-1}$ and we can write $\overline{\beta}: \Z \to G$ by $k \mapsto \psi \alpha^k \psi^{-1}$. Now define $\phi: G \rtimes_{\overline{\alpha}} \Z \to G \rtimes_{\overline{\beta}} \Z$ such that
    \[
        \phi(g,k) = (\psi(g), k).
    \]
    Note that $\phi$ is a bijection with inverse
    \[
        \phi^{-1}(g,k) = (\psi^{-1}(g),k).
    \]
    It remains to show that $\phi$ is a homomorphism. For any $g_1, g_2 \in G$ and $k_1, k_2 \in \Z$,
    \[
        \phi((g_1,k_1)(g_2,k_2))
        = \phi(g_1 \alpha^{k_1}(g_2), k_1+k_2)
        = (\psi(g_1 \alpha^{k_1}(g_2)), k_1+k_2)
    \]
    and
    \begin{align*}
        \phi(g_1, k_1) \phi(g_2, k_2)
        = (\psi(g_1), k_1)(\psi(g_2), k_2)
        &= (\psi(g_1) \beta^{k_1}(\psi(g_2)), k_1+k_2) \\
        &= (\psi(g_1) \psi(\alpha^{k_1}(g_2)), k_1+k_2) \\
        &= (\psi(g_1 \alpha^{k_1}(g_2)), k_1+k_2).
    \end{align*}
\end{proof}

\begin{remark}\label[remark]{isomorphism_class}
    As a consequence of \Cref{inn_isomorph} and \Cref{conjugation}, the isomorphism class of $H \rtimes_{\overline{\varphi}} \Z$ depends on
    \[
        [\varphi] = \Inn(H) \varphi \in \Aut(H)/\Inn(H) = \Out(H).
    \]
\end{remark}

Now, we construct an $n$-step poly-$\Z$ group $G_n$ inductively as follows. Let $G_0 = \{ 1 \}$ and $G_{i+1} = G_i \rtimes_{\overline{\varphi}_i} \Z$ for all $0 \leq i \leq n-1$ where $\varphi_i \in \Aut(G_i)$. Thus
\begin{equation*}
	G_n = (\dots ((\Z \rtimes_{\overline{\varphi}_1} \Z) \rtimes_{\overline{\varphi}_2} \Z) \rtimes_{\overline{\varphi}_3} \dots ) \rtimes_{\overline{\varphi}_{n-1}} \Z.
\end{equation*}
By \Cref{isomorphism_class}, at every step we choose a representative $\varphi_i$ of the class $[\varphi_i] \in \Out(G_i)$. For all $0 \leq i \leq n-1$, set $g_{i+1} = (1_{G_i}, 1)$. Then $\{G_0, \dots, G_n\}$ is the polycyclic series and $\{g_1, \dots, g_n\}$ is the polycyclic generating sequence of $G_n$. With the identification above, we have that for all $0 \leq i \leq n-1$, $\varphi_{i}^k(h) = g_{i+1}^khg_{i+1}^{-k}$ for all $h \in G_i$ and $k \in \Z$. We use the following $i \times i$ matrix notation for $\varphi_i$ for $1 \leq i \leq n-1$:
\[
    \varphi_i =
    \begin{bmatrix}
        a_{1,1} & \dots & a_{1, i} \\
        \vdots & \vdots & \vdots \\
        a_{i, 1} & \dots & a_{i, i}
    \end{bmatrix}
\]
where $\varphi_i(g_j) = g_1^{a_{1,j}}g_2^{a_{2,j}} \dots g_i^{a_{i,j}}$ for $j \leq i$.

We propose the following method for constructing a poly-$\Z$ group so as to increase the complexity of the multiplicative structure with each extension from $G_i$ to $G_{i+1}$. Construct an $n$-step poly-$\Z$ group as described above such that each $\varphi_i$ is chosen from $\Aut(G_i)$ such that $[\varphi_i] \in \Out(G_i)$ does not contain the identity element of $\Aut(G_i)$. It follows that $G_{i+1} = G_i \rtimes_{\overline{\varphi}_i} \Z \not\cong G_i \times \Z$ for all $0 \leq i \leq n-1$. Thus, the multiplicative structure of $G_{i+1}$ is more complicated than that of $G_i$, and we hypothesize that the conjugacy problem becomes increasingly complex with each extension.

\section{Constructions of poly-$\Z$ groups}\label{s_12stepgroups}
\subsection{$1$- and $2$-step poly-$\Z$ groups}
Up to isomorphism, there exists only one $1$-step poly-$\Z$ group $G_1 \cong \Z$. Since $\Aut(\Z) = \Out(\Z) = \left\{ [1], [-1]\right\}$, it follows that any $2$-step poly-$\Z$ group
\[
    G_2 \cong \Z \rtimes_{\overline{\varphi}} \Z =
    \begin{cases}
        \Z \times \Z &\text{ if } \varphi = [1] \\
        \langle g_1, g_2 \vert g_2g_1 = g_1^{-1}g_2 \rangle &\text{ if } \varphi = [-1].
    \end{cases}
\]
Note that when interpreted topologically, a $1$-step poly-$\Z$ group is isomorphic to the fundamental group of the circle, and a $2$-step poly-$\Z$ group is isomorphic to the fundamental group of either the torus (if abelian) or the Klein bottle (if not abelian).

\subsection{$3$-step poly-$\Z$ groups}
To construct a $3$-step poly-$\Z$ group $G_3$, we must first describe the automorphism group of $G_2$. In the case where $G_2 \cong \Z \times \Z$, we have that $\Aut(G_2) \cong \mathrm{GL}(2, \Z)$ and the reciprocal of \Cref{conjugation} is true \cite{gz}, \cite{n}. Thus, determining whether or not $(\Z \times \Z) \rtimes_{\overline{\varphi}} \Z$ and $(\Z \times \Z) \rtimes_{\overline{\varphi}'} \Z$ are isomorphic is equivalent to determining whether or not $\varphi, \varphi' \in \mathrm{GL}(2,\Z)$ are conjugate. To aid in identifying isomorphisms between these $3$-step poly-$\Z$ groups, there exists an algorithm for determining when two elements are conjugate in $\mathrm{GL}(2, \Z)$ \cite{ct}.

In the case where $G_2 = \langle g_1, g_2 \vert g_2g_1 = g_1^{-1}g_2 \rangle$, the description of $\Aut(G_2)$ is given by \Cref{automorphisms}, which relies on the following lemma. Throughout this paper, we define
\[
    \mu(x) =
    \begin{cases}
        0 &\text{ for } x \text{ even} \\
        1 &\text{ for } x \text{ odd}
    \end{cases}
\]
for any $x \in \Z$.

\begin{lemma}\label[lemma]{2stepexponents}
    Let $G_2$ be the $2$-step polycyclic group given by $\langle g_1, g_2 \mid g_2g_1 = g_1^{-1}g_2 \rangle$. Then the following hold:
    
    \begin{arabiclist}
        \item\label{2step:2} $g_2^ag_1^b = g_1^{b(-1)^a}g_2^a$
        \item\label{2step:3} $(g_1^{a}g_2^{b})^m = 
        \begin{cases}
            g_1^{ma}g_2^{mb} &\text{ for } b \text{ even} \\
    		g_1^{\mu(m)a}g_2^{mb} &\text{ for } b \text{ odd.}
        \end{cases}$
    \end{arabiclist}
\end{lemma}

\begin{proof}
    Notice that $G_2 \cong G_1 \rtimes_{\overline{\varphi}_1} \Z$ with $\varphi_1 \in \Aut(G_1)$ given by $\varphi_1(g_1) = g_2g_1g_2^{-1} = g_1^{-1}$. To prove part \ref{2step:2}, consider $g_2^ag_1^b$ in terms of $\varphi_1$:
    \[
        g_2^ag_1^b
        = g_2^ag_1^bg_2^{-a}g_2^a
        = \varphi_1^a(g_1^b)g_2^a
        = (\varphi_1^a(g_1))^bg_2^a
        = (g_1^{(-1)^a})^bg_2^a
        = g_1^{b(-1)^a}g_2^a.
    \]
    
    If $b$ is even, then it is clear that \ref{2step:3} holds for $m = 0, \pm1$. If $m < -1$, or $m = -l$ for some positive integer $l > 1$, then using \ref{2step:2},
    \[
        (g_1^{a}g_2^{b})^m
        = (g_1^{a}g_2^{b})^{-l}
        = ((g_1^{a}g_2^{b})^{-1})^l
        = (g_1^{-a}g_2^{-b})^l
        = (g_1^{-a}g_2^{-b})^{-m}.
    \]
    So this reduces to the case where $m > 1$, which we prove by induction on $m$. Using \ref{2step:2}, when $m=2$, $(g_1^{a}g_2^{b})^2 = g_1^{a}g_2^{b}g_1^{a}g_2^{b} = g_1^{2a}g_2^{2b}$. Assuming that $(g_1^{a}g_2^{b})^m = g_1^{ma}g_2^{mb}$, we have
    \[
    	(g_1^{a}g_2^{b})^{m+1}
    	= g_1^{ma}g_2^{mb} g_1^{a}g_2^{b}
    	= g_1^{(m+1)a}g_2^{(m+1)b}.
    \]
    
    Now, let $b$ be odd and $m$ be even. For $m=2$, using  \ref{2step:2}, we have
    \[
    	(g_1^{a}g_2^{b})^2 = g_1^{a}g_2^{b}g_1^{a}g_2^{b}
    	= g_2^{2 b}.
    \]
    For any even $m$, we have $m = 2k$ for some $k \in \Z$. Thus,
    \[
    	(g_1^{a}g_2^{b})^m = (g_1^{a}g_2^{b})^{2k}
    	= ((g_1^{a}g_2^{b})^2)^k
    	= (g_2^{2 b})^k
    	=  g_2^{2k b}
    	=  g_2^{mb}.
    \]
    
    Finally, let both $b$ and $m$ be odd, so $m = 2k+1$ for some $k \in \Z$. From part \ref{2step:2} and the fact that $(g_1^{a}g_2^{b})^{2k} = g_2^{2kb}$, we have
    \[
    	(g_1^{a}g_2^{b})^m = (g_1^{a}g_2^{b})^{2k+1}
    	= (g_1^{a}g_2^{b})^{2k}g_1^{a}g_2^{b}
    	= g_2^{2k b}g_1^{a}g_2^{b}
    	= g_1^{a}g_2^{2k b + b}
    	= g_1^{a}g_2^{mb}.
    \]
\end{proof}

Using the properties introduced in \Cref{2stepexponents}, we go on to define $\Aut(G_2)$.

\begin{theorem}\label{automorphisms}
    Let $G_2$ be the $2$-step polycyclic group given by $\langle g_1, g_2 \mid g_2g_1 = g_1^{-1}g_2 \rangle$. Then
    \[
        \Aut(G_2) = \left\{ \alpha_a, \beta_a, \gamma_a, \delta_a \mid a \in \Z \right\}
        \quad \text{ and } \quad
        \Inn(G_2) = \left \{ \beta_{2a}, \gamma_{2a}, \mid a \in \Z \right \}
    \]
    where
    \[
        \alpha_a = \begin{bmatrix} 1 & a \\ 0 & -1 \end{bmatrix},
        \beta_a = \begin{bmatrix} -1 & a \\ 0 & 1 \end{bmatrix},
        \gamma_a = \begin{bmatrix} 1 & a \\ 0 & 1 \end{bmatrix},
        \delta_a = \begin{bmatrix} -1 & a \\ 0 & -1 \end{bmatrix}
    \]
    with composition given by
    \begin{equation} \label{composition}
        \begin{tabular}{c|c c c c}
            $\circ$ & $\alpha_{a'}$ & $\beta_{b'}$ & $\gamma_{c'}$ & $\delta_{d'}$ \\
            \cline{1-5}
            $\alpha_a$ & $\gamma_{a+a'}$ & $\delta_{a+b'}$ & $\alpha_{a+c'}$ & $\beta_{a+d'}$ \\
            $\beta_b$ & $\delta_{b-a'}$ & $\gamma_{b-b'}$ & $\beta_{b-c'}$ & $\alpha_{b-d'}$ \\
            $\gamma_c$ & $\alpha_{c+a'}$ & $\beta_{c+b'}$ & $\gamma_{c+c'}$ & $\delta_{c+d'}$ \\
            $\delta_d$ & $\beta_{d-a'}$ & $\alpha_{d-b'}$ & $\delta_{d-c'}$ & $\gamma_{d-d'}$ \\
        \end{tabular}.
    \end{equation}
    In particular,
    \[
        \Aut(G_2) =
        \left\langle
            \alpha_0, \alpha_1, \beta_0, \beta_1
        \right\rangle
        \quad \text{ and } \quad
        \Inn(G_2) = \left\langle \alpha_1^2, \beta_0 \right\rangle.
    \]
\end{theorem}
    
\begin{proof}
    From the multiplicative structure defined, it follows that $G_2 = \langle g_1, g_2 \vert g_2g_1 = g_1^{-1}g_2 \rangle$. Let
    \[
        \psi = \begin{bmatrix} a & c \\ b & d \end{bmatrix} \in \Aut(G_2).
    \]
    Since the relation $\psi(g_2) = \psi(g_1)\psi(g_2)\psi(g_1)$ must be preserved and using \Cref{2stepexponents} we obtain
    \[
        g_1^c g_2^d
        = g_1^a g_2^b g_1^c g_2^d g_1^a g_2^b
        = g_1^{a + c (-1)^{b} + a (-1)^{d + b}} g_2^{2 b + d}.
    \]
    Thus
    \begin{align*}
        c &= a + c (-1)^{b} + a (-1)^{d + b}\\
        d &= 2 b + d.
    \end{align*}
    The second equality gives $b = 0$ and the first becomes $0 = a(1 + (-1)^d)$. We cannot have both $a, b = 0$, so we must have that $d$ is odd. If we set
    \[
        \psi^{-1} = \begin{bmatrix} a' & c' \\ b' & d' \end{bmatrix}
    \]
    then for the same reasons as above, $b'=0$ and $d'$ is odd. The equalities $g_1 = \psi(\psi^{-1}(g_1)) = g_1^{a a '}$ imply that $aa' = 1$.
    Therefore, $a = \pm 1$. Similarly, the equalities $g_2 = \psi(\psi^{-1}(g_2)) = g_1^{a c ' + c} g_2^{dd'}$ imply that $dd' = 1$. Therefore, $d = \pm 1$. It follows that $\psi \in \{ \alpha_c, \beta_c, \gamma_c, \delta_c\}$.
    
    Further, it is easy to check that $\alpha_a, \beta_a, \gamma_a, \delta_a \in \Aut(G_2)$ for all $a \in \Z$ with composition table given by \eqref{composition}. In particular, it follows that this group can be generated by $\alpha_0$, $\alpha_1$, $\beta_0$, and $\beta_1$, with identity element $\gamma_0$. Also note that
    \[
        \alpha_a^{-1} = \alpha_{-a},
        \quad \beta_a^{-1} = \beta_a,
        \quad \gamma_a^{-1} = \gamma_{-a},
        \quad \delta_a^{-1} = \delta_a.
    \]
    
    To identify $\Inn(G_2)$, let $i_h \in \Inn(G_2)$ where $h = g_1^ag_2^b$ for some $a, b \in \Z$. It follows that
    \[
        g_1^ag_2^b g_1 g_2^{-b}g_1^{-a} = g_1^{(-1)^b}
        \quad \text{ and } \quad
        g_1^ag_2^b g_2 g_2^{-b}g_1^{-a} = g_1^{2a} g_2.
    \]
    Therefore $i_h = \begin{bmatrix} \pm 1 & 2a \\ 0 & 1 \end{bmatrix}$ and $\Inn(G_2) = \left\{ \beta_{2b}, \gamma_{2c}, ~\vert~ b, c \in \Z \right\} = \left\langle \alpha_1^2, \beta_0 \right\rangle$.
\end{proof}

Using \Cref{automorphisms} and \cite{cs}, we identify all groups $G_3 = G_2 \rtimes_{\overline{\varphi}_2} \Z$ up to isomorphism as a corollary of the previous statement.

\begin{corollary}\label[corollary]{2stepautgroup}
    Let $G_2$ be the $2$-step polycyclic group given by $\langle g_1, g_2 \mid g_2g_1 = g_1^{-1}g_2 \rangle$. Then for every $\varphi_2 \in \Aut(G_2)$, there exists exactly one
    \[
        \varphi \in
            \left\{
                \begin{bmatrix} 1 & 0 \\ 0 & -1 \end{bmatrix},
                \begin{bmatrix} 1 & 1 \\ 0 & -1 \end{bmatrix},
                \begin{bmatrix} -1 & 0 \\ 0 & 1 \end{bmatrix},
                \begin{bmatrix} -1 & 1 \\ 0 & 1 \end{bmatrix}
            \right\}
    \]
    such that $G_2 \rtimes_{\overline{\varphi}_2} \Z \cong G_2 \rtimes_{\overline{\varphi}} \Z$.
\end{corollary}

\begin{proof}
    From \Cref{automorphisms}, we have that $\Aut(G_2)$ is made up of elements of forms
    \[
        \alpha_a = \begin{bmatrix} 1 & a \\ 0 & -1 \end{bmatrix},
        \beta_a = \begin{bmatrix} -1 & a \\ 0 & 1 \end{bmatrix},
        \gamma_a = \begin{bmatrix} 1 & a \\ 0 & 1 \end{bmatrix},
        \delta_a = \begin{bmatrix} -1 & a \\ 0 & -1 \end{bmatrix}
    \]
    for $a \in \Z$, where all elements in $\Inn(G_2)$ are of form $\beta_{2b}$ or $\gamma_{2c}$ for $b, c \in \Z$. From \Cref{inn_isomorph}, we identify all groups $G_2 \rtimes_{\overline{\varphi}_2} \Z$ up to isomorphism by determining which homomorphisms mapping $\Z$ to $\Aut(G_2)$ differ by $\beta_{2b}$ or $\gamma_{2c}$ for $b, c \in \Z$.
    
    Letting $\varphi \sim \varphi'$ indicate that $G_2 \rtimes_{\overline{\varphi}} \Z \cong G_2 \rtimes_{\overline{\varphi}'} \Z$, from the multiplicative structure of $\Aut(G_2)$ defined in \Cref{automorphisms},
    \begin{equation*}
        \begin{tabular}{c|c c c c}
            $\circ$ & $\alpha_b$ & $\beta_b$ & $\gamma_b$ & $\delta_b$  \\
            \cline{1-5}
            $\beta_{2a}$ & $\delta_{2a-b}$ & $\gamma_{2a-b}$ & $\beta_{2a-b}$ & $\alpha_{2a-b}$ \\
            $\gamma_{2a}$ & $\alpha_{2a+b}$ & $\beta_{2a+b}$ & $\gamma_{2a+b}$ & $\delta_{2a+b}$.
        \end{tabular}
    \end{equation*}
    So for any $\overline{\varphi} \colon \Z \to \Aut(G_2)$, we have that $\varphi \sim \alpha_0, \varphi \sim \alpha_1, \varphi \sim \beta_0, \text{ or } \varphi \sim \beta_1$.
    
    It remains to be shown that for any two $\varphi, \varphi' \in \{ \alpha_0, \alpha_1, \beta_0, \beta_1 \}$, $G_2 \rtimes_{\overline{\varphi}} \Z$ and $G_2 \rtimes_{\overline{\varphi}'} \Z$ are distinct up to isomorphism. Letting $G_{3,\varphi} = G_2 \rtimes_{\overline{\varphi}} \Z$, it can be shown that $Z(G_{3,\alpha_0}) = \langle g_3^2 \rangle$, $Z(G_{3,\alpha_1}) = \{ \mathrm{id} \}$, $Z(G_{3,\beta_0}) = \langle g_2^2, g_3^2 \rangle$, and $Z(G_{3,\beta_1}) = \langle g_2^2, g_3^2 \rangle$. Because isomorphic groups have isomorphic centers, it follows that $G_{3, \alpha_0}$ and $G_{3, \alpha_1}$ are distinct. Finally, by constructing an arbitrary map $\phi : G_{3,\beta_0} \to G_{3,\beta_1}$, it can be verified that $\phi$ can not be an isomorphism.
\end{proof}

\section{Descriptions of some $3$-step poly-$\Z$ groups} \label{s_3step}
In this section, we describe the structure of all $3$-step poly-$\Z$ groups $G_3 = G_2 \rtimes_{\overline{\varphi}_2} \Z$ up to isomorphism where $G_2$ is the group described in \Cref{2stepexponents}. For the group described in \Cref{defG3_b1}, we provide complete proofs for the convenience of the reader. For the remaining groups discussed in this section, the proofs of the multiplicative structure were omitted and those of the automorphism groups were shortened for the sake of conciseness.

\chunk{\label{defG3_b1}}
Consider
\begin{align*}
    G_1 &= \Z \\
    G_2 &= G_1 \rtimes_{\overline{\varphi}_1} \Z
    \quad \text{with} \quad
    \varphi_1 = [-1] \\
    G_3 &= G_2 \rtimes_{\overline{\varphi}_2} \Z
    \quad \text{with} \quad
    \varphi_2 = \begin{bmatrix} -1 & 1 \\ 0 & 1 \end{bmatrix}.
\end{align*}
Note that $G_3$ is well-defined as $\varphi_2 = \beta_1$ from \Cref{automorphisms}, and that
\[
    \langle g_1, g_2, g_3 \mid g_2g_1 = g_1^{-1}g_2, g_3g_1 = g_1^{-1}g_3, g_3g_2 = g_1g_2g_3 \rangle
\]
is a polycyclic presentation of $G_3$. We begin by defining the multiplicative structure of $G_3$.

\begin{lemma}\label[lemma]{b1_exponents}
    Let $G_3$ be the polycyclic group defined in \ref{defG3_b1}. Then the following hold:
    \begin{arabiclist}
        \item \label{b1:1} $g_i^ag_1^b = g_1^{b(-1)^a}g_i^a$ for $i = 2,3$
        \item \label{b1:2} $g_3^ag_2^b = g_1^{\mu(a)\mu(b)}g_2^bg_3^a$
        \item \label{b1:3} $(g_1^ag_2^bg_3^c)^m =
        \begin{cases}
            g_1^{ma}g_2^{mb}g_3^{mc} &\text{ for } b, c \text{ even} \\
            g_1^{ma - \lfloor m/2 \rfloor}g_2^{mb}g_3^{mc} &\text{ for } b, c \text{ odd} \\
            g_1^{\mu(m)a}g_2^{mb}g_3^{mc} &\text{ for } b+c \text{ odd.}
        \end{cases}$
    \end{arabiclist}
\end{lemma}

\begin{proof}
    Recall that $\varphi_2 \in \Aut(G_2)$ is given by $\varphi_2(g_1) = g_3g_1g_3^{-1} = g_1^{-1}$ and $\varphi_2(g_2) = g_3g_2g_3^{-1} = g_1g_2$. Relation \ref{b1:1} holds for $i = 2$ by \Cref{2stepexponents}. To show that it is true for $i = 3$, consider the multiplicative structure of $G_3$ in terms of $\varphi_2$:
    \[
        g_3^ag_1^b
        = g_3^ag_1^bg_3^{-a}g_3^a
        = \varphi_{2}^a(g_1^b)g_3^a
        = (\varphi_{2}^a(g_1))^bg_3^a
        = (g_1^{(-1)^a})^bg_3^a
        = g_1^{b(-1)^a}g_3^a.
    \]
    
    Similarly, we prove \ref{b1:2} by considering the multiplicative structure of $G_3$ in terms of $\varphi_2$ using the exponential rules defined in \Cref{2stepexponents}:
    \begin{align*}
        g_3^ag_2^b
        = g_3^ag_2^bg_3^{-a}g_3^a
        = \varphi_2^a(g_2^b)g_3^a
        = (\varphi_2^a(g_2))^bg_3^a
        = (g_1^{\mu(a)}g_2)^bg_3^a
        = g_1^{\mu(a)\mu(b)}g_2^bg_3^a.
    \end{align*}

    If $b, c$ are even, then it is clear that \ref{b1:3} holds for $m = 0, \pm 1$. For $m < -1$, or $m = -l$ for some positive integer $l > 1$, we have
    \[
        (g_1^ag_2^bg_3^c)^m 
        = (g_1^ag_2^bg_3^c)^{-l}
        = ((g_1^ag_2^bg_3^c)^{-1})^l
        = (g_1^{-a}g_2^{-b}g_3^{-c})^l
        = (g_1^{-a}g_2^{-b}g_3^{-c})^{-m}.
    \]
    So this reduces to the case where $m \geq 1$, which we prove by induction on $m$. Using \ref{b1:1} and \ref{b1:2}, when $m = 2$, $(g_1^ag_2^bg_3^c)^2 = g_1^ag_2^bg_3^cg_1^ag_2^bg_3^c = g_1^{2a}g_2^{2b}g_3^{2c}$. Now, assume that $(g_1^ag_2^bg_3^c)^m = g_1^{ma}g_2^{mb}g_3^{mc}$. Then
    \[
        (g_1^ag_2^bg_3^c)^{m+1}
        = g_1^{ma}g_2^{mb}g_3^{mc}g_1^ag_2^bg_3^c
        = g_1^{(m+1)a}g_2^{(m+1)b}g_3^{(m+1)c}.
    \]
    
    Now, let $b, c$ be odd; it is clear that \ref{b1:3} holds for $m = 0$ and $m = 1$. We prove that it holds for $m \geq 2$ by induction on $m$. When $m = 2$,
    \[
        (g_1^ag_2^bg_3^c)^2
        = g_1^ag_2^bg_3^cg_1^ag_2^bg_3^c
        = g_1^{2a}g_2^bg_3^cg_2^bg_3^c
        = g_1^{2a}g_2^bg_1g_2^bg_3^{2c}
        = g_1^{2a - 1} g_2^{2b}g_3^{2c}.
    \]
    Assume that $(g_1^ag_2^bg_3^c)^m = g_1^{ma - \lfloor m/2 \rfloor}g_2^{mb}g_3^{mc}$. If $m$ is even, then
    \begin{align*}
        (g_1^ag_2^bg_3^c)^{m+1}
        = g_1^{ma - m/2}g_2^{mb}g_3^{mc}g_1^ag_2^bg_3^c
        = g_1^{(m+1)a - m/2}g_2^{(m+1)b}g_3^{(m+1)c},
    \end{align*}
    and if $m$ is odd, then
    \begin{align*}
        (g_1^ag_2^bg_3^c)^{m+1}
        = g_1^{ma - (m-1)/2}g_2^{mb}g_3^{mc}g_1^ag_2^bg_3^c
        &= g_1^{(m+1)a - (m-1)/2}g_2^{mb}g_3^{mc}g_2^bg_3^c \\
        &= g_1^{(m+1)a - (m-1)/2}g_2^{mb}g_1g_2^bg_3^{(m+1)c} \\
        &= g_1^{(m+1)a - (m+1)/2}g_2^{(m+1)b}g_3^{(m+1)c}.
    \end{align*}
    If $m \leq -1$, then $m = -l$ for some positive integer $l$, and from above,
    \begin{align*}
        (g_1^ag_2^bg_3^c)^m
        = ((g_1^ag_2^bg_3^c)^{-1})^l
        = (g_1^{1-a}g_2^{-b}g_3^{-c})^l
        = g_1^{-la + l - \lfloor l/2 \rfloor}g_2^{-lb}g_3^{-lc}
        &= g_1^{-la - \lfloor -l/2 \rfloor}g_2^{-lb}g_3^{-lc} \\
        &= g_1^{ma - \lfloor m/2 \rfloor}g_2^{mb}g_3^{mc}.
    \end{align*}
    
    Finally, let $b$ and $c$ have different parity. First, suppose $m$ is even, so $m = 2k$ for some $k \in \Z$. Then
    \[
        (g_1^ag_2^bg_3^c)^m
        = ((g_1^ag_2^bg_3^c)^2)^k
        = (g_2^{2b}g_3^{2c})^k
        = g_2^{2kb}g_3^{2kc}.
    \]
    Now, suppose $m$ is odd, so $m = 2k+1$ for some $k \in \Z$. Using the fact from above that $(g_1^ag_2^bg_3^c)^{2k} = g_2^{2kb}g_3^{2kc}$,
    \[
        (g_1^ag_2^bg_3^c)^m
        = (g_1^ag_2^bg_3^c)^{2k}g_1^ag_2^bg_3^c
        = g_2^{2kb}g_3^{2kc}g_1^ag_2^bg_3^c
        = g_1^ag_2^{(2k+1)b}g_3^{(2k+1)c}.
    \]
\end{proof}

Using the relations defined in \Cref{b1_exponents}, we go on to define $\Aut(G_3)$. 

\begin{theorem}\label{b1_automorphisms}
    Let $G_3 = \langle g_1, g_2, g_3 \mid g_2g_1 = g_1^{-1}g_2, g_3g_1 = g_1^{-1}g_3, g_3g_2 = g_1g_2g_3 \rangle$ be the polycyclic group defined in \ref{defG3_b1}. Set
    \begin{align*}
        \mathcal{A} &= \left\{
            A \left|
            A = 
            \begin{bmatrix}
                2k & 2m+1 \\
                2l+1 & 2n
            \end{bmatrix}
            \right.
            \in \mathrm{GL}(2, \Z) \text{ where } k,l,m,n \in \Z
            \right\}, \\
        \mathcal{B} &= \left\{
            B \left|
            B =
            \begin{bmatrix}
                2k+1 & 2m \\
                2l & 2n+1
            \end{bmatrix}
            \right.
            \in \mathrm{GL}(2, \Z) \text{ where } k,l,m,n \in \Z
        \right\},
    \end{align*}
    and for $a \in \Z$, $A \in \mathcal{A}$, and $B \in \mathcal{B}$, set
    \begin{align*}
        \alpha_{a, A} &=
        \left[
            \begin{array}{cc}
                1 & \begin{matrix} a & a+1 \end{matrix} \\
                \begin{matrix} 0 \\ 0 \end{matrix} & A
            \end{array}
        \right],
        &\beta_{a, A} &=
        \left[
            \begin{array}{cc}
                -1 & \begin{matrix} a & a \end{matrix} \\
                \begin{matrix} 0 \\ 0 \end{matrix} & A
            \end{array}
        \right], \\
        \gamma_{a, B} &=
        \left[
            \begin{array}{cc}
                1 & \begin{matrix} a & a \end{matrix} \\
                \begin{matrix} 0 \\ 0 \end{matrix} & B
            \end{array}
        \right],
        &\delta_{a, B} &=
        \left[
            \begin{array}{cc}
                -1 & \begin{matrix} a & a-1 \end{matrix} \\
                \begin{matrix} 0 \\ 0 \end{matrix} & B
            \end{array}
        \right].
    \end{align*}
    Then
    \begin{align*}
        \Aut(G_3) &= \{ \alpha_{a,A}, \beta_{a,A}, \gamma_{a,B}, \delta_{a,B} \mid a \in \Z, A \in \mathcal{A}, B \in \mathcal{B} \} \\
        \Inn(G_3) &= \{ \gamma_{a, I_2}, \delta_{a, I_2} \mid a \in \Z\} \\
        \Out(G_3) &= \{[\alpha_{0,A}], [\gamma_{0,B}] \mid A \in \mathcal{A}, B \in \mathcal{B} \}.
    \end{align*}
    Moreover, the composition table of $\Out(G_3)$ is given by
    \begin{equation*}
        \begin{tabular}{c|c c}
            $\circ$ & $[\alpha_{0,A'}]$ & $[\gamma_{0,B'}]$ \\
            \cline{1-3}
            $[\alpha_{0,A}]$ & $[\gamma_{0,AA'}]$ & $[\alpha_{0,AB'}]$ \\
            $[\gamma_{0,B}]$ & $[\alpha_{0,BA'}]$ & $[\gamma_{0,BB'}]$ \\
        \end{tabular}
    \end{equation*}
    for $A, A' \in \mathcal{A}$ and $B,B' \in \mathcal{B}$.
\end{theorem}

\begin{proof}
    Let
    \[
        \psi =
        \begin{bmatrix}
            a & d & g \\
            b & e & h \\
            c & f & i
        \end{bmatrix}
        \in \Aut(G_3).
    \]
    Using the multiplicative rules described in \Cref{b1_exponents}, we consider the relations of $G_3$. Since the relations $\psi(g_2)\psi(g_1) = \psi(g_1)^{-1}\psi(g_2)$ must be preserved, we obtain
    \begin{align*}
        g_1^{d+a(-1)^{f+e} + \mu(b) \mu(f)(-1)^e}g_2^{b+e}g_3^{c+f}
        &= g_1^{(d-a)(-1)^{-b-c} + \mu(e-b) \mu(c)}g_2^{e-b}g_3^{f-c}.
    \end{align*}
    Thus, $b=c=0$ and it follows that $a(-1)^{e+f} = -a$ and therefore $e$ and $f$ must have different parity.
    
    Since the relation $\psi(g_3)\psi(g_1) = \psi(g_1)^{-1}\psi(g_3)$ must be preserved, we obtain
    \[
        g_1^{a(-1)^{h+i}}g_2^hg_3^i
        = g_1^{-a}g_2^hg_3^i
    \]
    and it follows that $h$ and $i$ must have different parity.
    
    Since the relation $\psi(g_3)\psi(g_2) = \psi(g_1)\psi(g_2)\psi(g_3)$ must be preserved, we obtain
    \[
        g_1^{g-d+ \mu(e) \mu(i)(-1)^h}g_2^{e+h}g_3^{f+i}
        = g_1^{a+d-g + \mu(f) \mu(h)(-1)^e}g_2^{e+h}g_3^{f+i},
    \]
    so we have that
    \[
        \mu(e) \mu(i)(-1)^h - \mu(f) \mu(h)(-1)^e = a + 2d - 2g.
    \]
    If we set
    \[
        \psi^{-1} =
        \begin{bmatrix}
            a' & d' & g' \\
            b' & e' & h' \\
            c' & f' & i'
        \end{bmatrix},
    \]
    then for the same reasons as above, we must have that $b'=c'=0$, $e'$ and $f'$ have different parity, and $h'$ and $i'$ have different parity. For any $\psi$, the following identities must also hold:
    \begin{align}
        \psi^{-1}(\psi(g_1)) &= g_1 \label{b1_g1id} \\
        \psi^{-1}(\psi(g_2)) &= g_2 \label{b1_g2id} \\
        \psi^{-1}(\psi(g_3)) &= g_3 \label{b1_g3id}
    \end{align}
    If $e$ and $h$ have the same parity, it becomes clear that there exist no $d, g, a' \in \Z$ such that \eqref{b1_g1id} is true, so we must have that $e$ and $h$ have different parity (and thus $f$ and $i$ must have different parity as well). Then, in any case, we have that $a = -2(d-g) \pm 1$, so from \eqref{b1_g1id}, $a = a' = \pm 1$.
    
    If $e, i$ are even and $f, h$ are odd, then $a = a' = 1$ and $g = d+1$ or $a = a' = -1$ and $d = g$. If $e, i$ are odd and $f, h$ are even, then $a = a' = 1$ and $d = g$ or $a = a' = -1$ and $g = d-1$. Thus, letting
    \[
        A =
        \begin{bmatrix}
            2k & 2m+1 \\
            2l+1 & 2n
        \end{bmatrix}
        \quad
        \text{ and }
        \quad
        B =
        \begin{bmatrix}
            2k+1 & 2m \\
            2l & 2n+1
        \end{bmatrix}
    \]
    such that $k,l,m,n \in \Z$ and $A,B \in \mathrm{GL}(2, \Z)$, we must have that $\psi \in \{\alpha_{d, A}, \beta_{d, A}, \gamma_{d, B}, \delta_{d, B}\}$. It can be easily verified that these values of $\psi$ are consistent with both \eqref{b1_g2id} and \eqref{b1_g3id} as well.
    
    Further, it is easy to check that $\alpha_{a, A}, \beta_{a, A}, \gamma_{a, B}, \delta_{a, B} \in \Aut(G_3)$ for all $a \in \Z$ and all $A, B$ as defined above, and that
    \[
        \alpha_{a,A}^{-1} = \alpha_{-a-1,A^{-1}}, \quad
        \beta_{a,A}^{-1} = \beta_{a,A^{-1}}, \quad
        \gamma_{a,B}^{-1} = \gamma_{-a,B^{-1}}, \quad
        \delta_{a,B}^{-1} = \delta_{a,B^{-1}}.
    \]
    
    To identify $\Inn(G_3)$, let $\iota_h \in \Inn(G_3)$ where $h = g_1^ag_2^bg_3^c$ for some $a, b, c \in \Z$. It follows that
    \begin{align*}
        g_1^ag_2^bg_3^cg_1(g_1^ag_2^bg_3^c)^{-1}
        &= g_1^{(-1)^{b+c}} \\
        g_1^ag_2^bg_3^cg_2(g_1^ag_2^bg_3^c)^{-1}
        &= g_1^{2a + \mu(c)(-1)^b}g_2 \\
        g_1^ag_2^bg_3^cg_3(g_1^ag_2^bg_3^c)^{-1}
        &= g_1^{2a - \mu(b)}g_3.
    \end{align*}
    Therefore,
    \[
        \iota_h \in
        \left\{
            \begin{bmatrix}
                1 & 2a-1 & 2a-1 \\
                0 & 1 & 0 \\
                0 & 0 & 1
            \end{bmatrix},
            \begin{bmatrix}
                -1 & 2a & 2a-1 \\
                0 & 1 & 0 \\
                0 & 0 & 1
            \end{bmatrix},
            \begin{bmatrix}
                -1 & 2a+1 & 2a \\
                0 & 1 & 0 \\
                0 & 0 & 1
            \end{bmatrix},
            \begin{bmatrix}
                1 & 2a & 2a \\
                0 & 1 & 0 \\
                0 & 0 & 1
            \end{bmatrix}
        \right\}
    \]
    and $\Inn(G_3) = \{ \gamma_{a, I_2}, \delta_{a, I_2} \mid a \in \Z\}$. It can be verified that for any $a,b \in \Z$ and any $A, B$ as defined above,
    \begin{equation*}
        \begin{tabular}{c|c c c c}
            $\circ$ & $\alpha_{b,A}$ & $\beta_{b,A}$ & $\gamma_{b,B}$ & $\delta_{b,B}$  \\
            \cline{1-5}
            $\gamma_{a,I_2}$ & $\alpha_{a+b,A}$ & $\beta_{a+b,A}$ & $\gamma_{a+b,B}$ & $\delta_{a+b,B}$ \\
            $\delta_{a,I_2}$ & $\beta_{a-b,A}$ & $\alpha_{a-b,A}$ & $\delta_{a-b,B}$ & $\gamma_{a-b,B}$
        \end{tabular},
    \end{equation*}
    and thus $\Out(G_3) = \langle [\alpha_{0,A}], [\gamma_{0,B}]\rangle$ with composition structure
    \begin{equation*}
        \begin{tabular}{c|c c}
            $\circ$ & $[\alpha_{0,A'}]$ & $[\gamma_{0,B'}]$ \\
            \cline{1-3}
            $[\alpha_{0,A}]$ & $[\gamma_{0,AA'}]$ & $[\alpha_{0,AB'}]$ \\
            $[\gamma_{0,B}]$ & $[\alpha_{0,BA'}]$ & $[\gamma_{0,BB'}]$ \\
        \end{tabular}.
    \end{equation*}
\end{proof}

\chunk{\label{defG3_a0}}
Consider
\begin{align*}
    G_1 &= \Z \\
    G_2 &= G_1 \rtimes_{\overline{\varphi}_1} \Z
    \quad \text{with} \quad
    \varphi_1 = [-1] \\
    G_3 &= G_2 \rtimes_{\overline{\varphi}_2} \Z
    \quad \text{with} \quad
    \varphi_2 = \begin{bmatrix} 1 & 0 \\ 0 & -1 \end{bmatrix}.
\end{align*}
Note that $G_3$ is well-defined as $\varphi_2 = \alpha_0$ from \Cref{automorphisms}, and that
\[
    \langle g_1, g_2, g_3 \mid g_2g_1 = g_1^{-1}g_2, g_3g_1 = g_1g_3, g_3g_2 = g_2^{-1}g_3 \rangle
\]
is a polycyclic presentation of $G_3$. We begin by defining the multiplicative structure of $G_3$.

\begin{lemma} \label[lemma]{a0_exponents}
    Let $G_3$ be the polycyclic group defined in \ref{defG3_a0}. Then the following hold:
    \begin{arabiclist}
        \item \label{a0:1} $g_i^ag_{i-1}^b = g_{i-1}^{b(-1)^a}g_i^a$ for $i = 2,3$
        \item \label{a0:2} $(g_1^ag_2^bg_3^c)^m =
        \begin{cases}
            g_1^{ma}g_2^{mb}g_3^{mc} &\text{ if } b,c \text{ even} \\
            g_1^{ma}g_2^{\mu(m)b}g_3^{mc} &\text{ if } b \text{ even, } c \text{ odd} \\
            g_1^{\mu(m)a}g_2^{mb}g_3^{mc} &\text{ if } b \text{ odd, } c \text{ even} \\
            g_1^{\mu(m)a}g_2^{\mu(m)b}g_3^{mc} &\text{ if } b,c \text{ odd.}
        \end{cases}$
    \end{arabiclist}
\end{lemma}

Using the relations defined in \Cref{a0_exponents}, we go on to define $\Aut(G_3)$.

\begin{theorem} \label{a0_automorphisms}
    Let $G_3 = \langle g_1, g_2, g_3 \mid g_2g_1 = g_1^{-1}g_2, g_3g_1 = g_1g_3, g_3g_2 = g_2^{-1}g_3 \rangle$ be the polycyclic group defined in \ref{defG3_a0}. For $a, b, c \in \Z$, set
    \begin{align*}
        \alpha_{a, b, c} &=
        \begin{bmatrix}
            1 & a & 0 \\
            0 & 1 & 2b \\
            0 & 0 & (-1)^c
        \end{bmatrix},
        &\beta_{a, b, c} &=
        \begin{bmatrix}
            1 & a & 0 \\
            0 & -1 & 2b \\
            0 & 0 & (-1)^c
        \end{bmatrix}, \\
        \gamma_{a, b, c} &=
        \begin{bmatrix}
            -1 & a & 0 \\
            0 & 1 & 2b \\
            0 & 0 & (-1)^c
        \end{bmatrix},
        &\delta_{a, b, c} &=
        \begin{bmatrix}
            -1 & a & 0 \\
            0 & -1 & 2b \\
            0 & 0 & (-1)^c
        \end{bmatrix}.
    \end{align*}
    Then
    \begin{align*}
        \Aut(G_3) &= \{\alpha_{a, b, c}, \beta_{a, b, c}, \gamma_{a, b, c}, \delta_{a, b, c} \mid a, b \in \Z, c \in \Z_2 \} \\
        \Inn(G_3) &= \{\alpha_{2a, b, 0}, \beta_{2a, b, 0}, \gamma_{2a, b, 0}, \delta_{2a, b, 0} \mid a, b \in \Z \} \\
        \Out(G_3) &= \{[\alpha_{0,0,0}], [\alpha_{1,0,0}], [\alpha_{0,0,1}], [\alpha_{1,0,1}] \}.
    \end{align*}
    Moreover, the composition table of $\Out(G_3)$ is given by
    \begin{equation*}
        \begin{tabular}{c| c c}
            $\circ$ & $[\alpha_{0,0,b}]$ & $[\alpha_{1,0,b}]$  \\
            \cline{1-3}
            $[\alpha_{0,0,a}]$ & $[\alpha_{0,0,a+b}]$ & $[\alpha_{1,0,a+b}]$ \\
            $[\alpha_{1,0,a}]$ & $[\alpha_{1,0,a+b}]$ & $[\alpha_{0,0,a+b}]$
        \end{tabular}
    \end{equation*}
    for $a,b \in \Z_2$.
\end{theorem}

\begin{proof}
    Let
    \[
        \psi =
        \begin{bmatrix}
            a & d & g \\
            b & e & h \\
            c & f & i
        \end{bmatrix}
        \in \Aut(G_3)
        \quad \text{ with } \quad
        \psi^{-1} =
        \begin{bmatrix}
            a' & d' & g' \\
            b' & e' & h' \\
            c' & f' & i'
        \end{bmatrix}.
    \]
    As the relations of $G_3$ must be preserved, we obtain $b, b', c, c', f, f', g, g' = 0$, $e, e', i, i'$ are odd, and $h, h'$ are even. From the identities $g_1 = \psi(\psi^{-1}(g_1))$, $g_2 = \psi(\psi^{-1}(g_2))$, and $g_3 = \psi(\psi^{-1}(g_3))$, we obtain $a = a' = \pm 1$, $e = e' = \pm 1$, $i = i' = \pm 1$. It also follows that if $a = 1$ then $d = -d'$ while if $a = -1$ then $d = d'$, and that if $e = 1$ then $h = -h'$ while if $e = -1$ then $h = h'$. So we have $\psi \in \{\alpha_{a,b,c}, \beta_{a,b,c}, \gamma_{a,b,c}, \delta_{a,b,c}\}$ for $a,b,c \in \Z$.
    
    Further, it is easy to check that $\alpha_{a,b,c}, \beta_{a,b,c}, \gamma_{a,b,c}, \delta_{a,b,c} \in \Aut(G_3)$ for all $a, b, c \in \Z$ and that
    \[
        \alpha_{a,b,c}^{-1} = \alpha_{-a,-b,c}, \quad \beta_{a,b,c}^{-1} = \beta_{-a,b,c}, \quad \gamma_{a,b,c}^{-1} = \gamma_{a,-b,c}, \quad \delta_{a,b,c}^{-1} = \delta_{a,b,c}.
    \]
    
    To identify $\Inn(G_3)$, let $\iota_h \in \Inn(G_3)$ where $h = g_1^ag_2^bg_3^c$ for some $a, b, c \in \Z$. It follows that
    \begin{align*}
        g_1^ag_2^bg_3^cg_1(g_1^ag_2^bg_3^c)^{-1}
        &= g_1^{(-1)^b} \\
        g_1^ag_2^bg_3^cg_2(g_1^ag_2^bg_3^c)^{-1}
        &= g_1^{2a}g_2^{(-1)^c} \\
        g_1^ag_2^bg_3^cg_3(g_1^ag_2^bg_3^c)^{-1}
        &= g_2^{2b}g_3.
    \end{align*}
    Therefore, $\Inn(G_3) = \{ \alpha_{2a,b,0}, \beta_{2a,b,0}, \gamma_{2a,b,0}, \delta_{2a,b,0} \}$ and it can be verified that for any $a, b, c, d, e \in \Z$,
    \begin{equation*}
        \begin{tabular}{c|c c c c}
            $\circ$ & $\alpha_{c,d,e}$ & $\beta_{c,d,e}$ & $\gamma_{c,d,e}$ & $\delta_{c,d,e}$  \\
            \cline{1-5}
            $\alpha_{2a,b,0}$ & $\alpha_{2a+c,b+d,e}$ & $\beta_{2a+c,b+d,e}$ & $\gamma_{2a+c,b+d,e}$ & $\delta_{2a+c,b+d,e}$ \\
            $\beta_{2a,b,0}$ & $\beta_{2a+c,b-d,e}$ & $\alpha_{2a+c,b-d,e}$ & $\delta_{2a+c,b-d,e}$ & $\gamma_{2a+c,b-d,e}$ \\
            $\gamma_{2a,b,0}$ & $\gamma_{2a-c,b+d,e}$ & $\delta_{2a-c,b+d,e}$ & $\alpha_{2a-c,b+d,e}$ & $\beta_{2a-c,b+d,e}$ \\
            $\delta_{2a,b,0}$ & $\delta_{2a-c,b-d,e}$ & $\gamma_{2a-c,b-d,e}$ & $\beta_{2a-c,b-d,e}$ & $\alpha_{2a-c,b-d,e}$
        \end{tabular},
    \end{equation*}
    and thus $\Out(G_3) = \{[\alpha_{0,0,0}], [\alpha_{1,0,0}], [\alpha_{0,0,1}], [\alpha_{1,0,1}]\}$ with composition structure
    \begin{equation*}
        \begin{tabular}{c| c c}
            $\circ$ & $[\alpha_{0,0,b}]$ & $[\alpha_{1,0,b}]$  \\
            \cline{1-3}
            $[\alpha_{0,0,a}]$ & $[\alpha_{0,0,a+b}]$ & $[\alpha_{1,0,a+b}]$ \\
            $[\alpha_{1,0,a}]$ & $[\alpha_{1,0,a+b}]$ & $[\alpha_{0,0,a+b}]$
        \end{tabular}
    \end{equation*}
    where $a,b \in \Z_2$.
\end{proof}

\chunk{\label{defG3_a1}}
Consider
\begin{align*}
    G_1 &= \Z \\
    G_2 &= G_1 \rtimes_{\overline{\varphi}_1} \Z
    \quad \text{with} \quad
    \varphi_1 = [-1] \\
    G_3 &= G_2 \rtimes_{\overline{\varphi}_2} \Z
    \quad \text{with} \quad
    \varphi_2 = \begin{bmatrix} 1 & 1 \\ 0 & -1 \end{bmatrix}.
\end{align*}
Note that $G_3$ is well-defined as $\varphi_2 = \alpha_1$ from \Cref{automorphisms}, and that
\[
    \langle g_1, g_2, g_3 \mid g_2g_1 = g_1^{-1}g_2, g_3g_1 = g_1g_3, g_3g_2 = g_1g_2^{-1}g_3 \rangle
\]
is a polycyclic presentation of $G_3$. We begin by defining the multiplicative structure of $G_3$.

\begin{lemma} \label[lemma]{a1_exponents}
    Let $G_3$ be the polycyclic group defined in \ref{defG3_a1}. Then the following hold:
    \begin{arabiclist}
        \item \label{a1:1} $g_2^ag_1^b = g_1^{b(-1)^a}g_2^a$
        \item \label{a1:2} $g_3^ag_2^b = g_1^{\mu(b)a}g_2^{b(-1)^a}g_3^a$
        \item \label{a1:3} $(g_1^ag_2^bg_3^c)^m =
        \begin{cases}
            g_1^{ma}g_2^{mb}g_3^{mc} &\text{ for } b, c \text{ even} \\
            g_1^{ma}g_2^{\mu(m)b}g_3^{mc} &\text{ for } b \text{ even, } c \text{ odd} \\
            g_1^{\mu(m)a + \lfloor m/2 \rfloor c (-1)^{m+1}}g_2^{mb}g_3^{mc} &\text{ for } b \text{ odd, } c \text{ even} \\
            g_1^{\mu(m)a + \lfloor m/2 \rfloor c (-1)^{m+1}}g_2^{\mu(m)b}g_3^{mc} &\text{ for } b,c \text{ odd}.
        \end{cases}$
    \end{arabiclist}
\end{lemma}

Using the relations defined in \Cref{a1_exponents}, we go on to define $\Aut(G_3)$.

\begin{theorem} \label{a1_automorphisms}
    Let $G_3 = \langle g_1, g_2, g_3 \mid g_2g_1 = g_1^{-1}g_2, g_3g_1 = g_1g_3, g_3g_2 = g_1g_2^{-1}g_3 \rangle$ be the polycyclic group defined in \Cref{defG3_a1}. For $a,b,c,d \in \Z$, set
    \begin{align*}
        \alpha_{a,b,c,d} &=
        \begin{bmatrix}
            1 & a & b \\
            0 & 1 & 2c \\
            0 & 0 & (-1)^d
        \end{bmatrix},
        &\beta_{a,b,c,d} &=
        \begin{bmatrix}
            1 & a & b \\
            0 & -1 & 2c \\
            0 & 0 & (-1)^d
        \end{bmatrix}, \\
        \gamma_{a,b,c,d} &=
        \begin{bmatrix}
            -1 & a & b \\
            0 & 1 & 2c \\
            0 & 0 & (-1)^d
        \end{bmatrix},
        &\delta_{a,b,c,d} &=
        \begin{bmatrix}
            -1 & a & b \\
            0 & -1 & 2c \\
            0 & 0 & (-1)^d
        \end{bmatrix}.
    \end{align*}
    Then
    \begin{align*}
        \Aut(G_3) &= \{\alpha_{a,b,c,d}, \beta_{a,b,c,d}, \gamma_{a,b,c,d}, \delta_{a,b,c,d} \mid a, b, c \in \Z, d \in \Z_2\} \\
        \Inn(G_3) &= \{\alpha_{2a,0,b,0}, \beta_{2a+1,0,b,0}, \gamma_{2a,-1,b,0}, \delta_{2a+1,-1,b,0} \mid a, b \in \Z \} \\
        \Out(G_3) &= \{[\alpha_{0,a,0,0}], [\alpha_{0,a,0,1}], [\alpha_{1,a,0,0}], [\alpha_{1,a,0,1}] \mid a \in \Z \}.
    \end{align*}
    Moreover, the composition table of $\Out(G_3)$ is given by
    \begin{equation*}
        \begin{tabular}{c|c c}
            $\circ$ & $[\alpha_{0,c,0,d}]$ & $[\alpha_{1,c,0,d}]$ \\
            \cline{1-3}
            $[\alpha_{0,a,0,b}]$ & $[\alpha_{0,c+(-1)^da,0,b+d}]$ & $[\alpha_{1,c+(-1)^da,0,b+d}]$ \\
            $[\alpha_{1,a,0,b}]$ & $[\alpha_{1,c+(-1)^da,0,b+d}]$ & $[\alpha_{0,c+(-1)^da,0,b+d}]$
        \end{tabular}
    \end{equation*}
    for $a,c \in \Z$ and $b,d \in \Z_2$.
\end{theorem}

\begin{proof}
    Let
    \[
        \psi =
        \begin{bmatrix}
            a & d & g \\
            b & e & h \\
            c & f & i
        \end{bmatrix}
        \in \Aut(G_3)
        \quad \text{ with } \quad
        \psi^{-1} =
        \begin{bmatrix}
            a' & d' & g' \\
            b' & e' & h' \\
            c' & f' & i'
        \end{bmatrix}.
    \]
    As the relations of $G_3$ must be preserved, we obtain $b, b', c, c', f, f' = 0$, $e, e', i, i'$ are odd, and $h, h'$ are even. From the identities $g_1 = \psi(\psi^{-1}(g_1))$, $g_2 = \psi(\psi^{-1}(g_2))$, and $g_3 = \psi(\psi^{-1}(g_3))$, we also obtain $a = a' = \pm 1$, $e = e' = \pm 1$, and $i = i' = \pm 1$. It also follows that if $a = 1$ then $d' = -d$ while if $a = -1$ then $d' = d$, if $e = 1$ then $h' = -h$ while if $e = -1$ then $h' = h$, and if $a = i$ then $g' = -g$, while if $a \neq i$ then $g' = g$. Thus, for $a, b, c, d \in \Z$, we must have that $\psi \in \{\alpha_{a,b,c,d}, \beta_{a,b,c,d}, \gamma_{a,b,c,d}, \delta_{a,b,c,d}\}$.
    
    Further, it is easy to verify that $\alpha_{a,b,c,d}, \beta_{a,b,c,d}, \gamma_{a,b,c,d}, \delta_{a,b,c,d} \in \Aut(G_3)$ for all $a, b, c, d \in \Z$, and that
    \begin{align*}
        \alpha_{a,b,c,d_1}^{-1} &= \alpha_{-a,-b,-c,d_1},
        &\alpha_{a,b,c,d_2}^{-1} &= \alpha_{-a,b,-c,d_2},
        &\beta_{a,b,c,d_1}^{-1} &= \beta_{-a,-b,c,d_1},
        &\beta_{a,b,c,d_2}^{-1} &= \beta_{-a,b,c,d_2}, \\
        \gamma_{a,b,c,d_1}^{-1} &= \gamma_{a,b,-c,d_1},
        &\gamma_{a,b,c,d_2}^{-1} &= \gamma_{a,-b,-c,d_2},
        &\delta_{a,b,c,d_1}^{-1} &= \delta_{a,b,c,d_1},
        &\delta_{a,b,c,d_2}^{-1} &= \delta_{a,-b,c,d_2}
    \end{align*}
    where $d_1$ is even and $d_2$ is odd.
    
    To identify $\Inn(G_3)$, let $\iota_h \in \Inn(G_3)$ where $h = g_1^ag_2^bg_3^c$ for some $a, b, c \in \Z$. It follows that
    \begin{align*}
        g_1^ag_2^bg_3^cg_1(g_1^ag_2^bg_3^c)^{-1}
        &= g_1^{(-1)^b} \\
        g_1^ag_2^bg_3^cg_2(g_1^ag_2^bg_3^c)^{-1}
        &= g_1^{2a+c(-1)^b}g_2^{(-1)^c} \\
        g_1^ag_2^bg_3^cg_3(g_1^ag_2^bg_3^c)^{-1}
        &= g_1^{-\mu(b)}g_2^{2b}g_3.
    \end{align*}
    Therefore, $\Inn(G_3) = \{\alpha_{2a,0,b,0}, \beta_{2a+1,0,b,0}, \gamma_{2a,-1,b,0}, \delta_{2a+1,-1,b,0}\}$ and it can be verified that for any $a,b,c,d,e,f \in \Z$,
    \begin{equation*}
        {\scriptsize
        \begin{tabular}{c|c c c c}
            $\circ$ & $\alpha_{c,d,e,f}$ & $\beta_{c,d,e,f}$ & $\gamma_{c,d,e,f}$ & $\delta_{c,d,e,f}$ \\
            \cline{1-5}
            $\alpha_{2a,0,b,0}$ & $\alpha_{2a+c,d,b+e,f}$ & $\beta_{2a+c,d,b+e,f}$ & $\gamma_{2a+c,d,b+e,f}$ & $\delta_{2a+c,d,b+e,f}$ \\
            $\beta_{2a+1,0,b,0}$ & $\beta_{2a+c+1,d,b-e,f}$ & $\alpha_{2a+c+1,d,b-e,f}$ & $\delta_{2a+c+1,d,b-e,f}$ & $\gamma_{2a+c+1,d,b-e,f}$ \\
            $\gamma_{2a,-1,b,0}$ & $\gamma_{2a-c,-d-(-1)^f,b+e,f}$ & $\delta_{2a-c,-d-(-1)^f,b+e,f}$ & $\alpha_{2a-c,-d-(-1)^f,b+e,f}$ & $\beta_{2a-c,-d-(-1)^f,b+e,f}$ \\
            $\delta_{2a+1,-1,b,0}$ & $\delta_{2a-c+1,d-(-1)^f,b-e,f}$ & $\gamma_{2a-c+1,d-(-1)^f,b-e,f}$ & $\beta_{2a-c+1,d-(-1)^f,b-e,f}$ & $\alpha_{2a-c+1,d-(-1)^f,b-e,f}$
        \end{tabular},
        }
    \end{equation*}
    and thus
    \[
        \Out(G_3) = \{[\alpha_{0,a,0,0}], [\alpha_{0,a,0,1}], [\alpha_{1,a,0,0}], [\alpha_{1,a,0,1}] \}
    \]
    with composition structure
    \begin{equation*}
        \begin{tabular}{c|c c}
            $\circ$ & $[\alpha_{0,c,0,d}]$ & $[\alpha_{1,c,0,d}]$ \\
            \cline{1-3}
            $[\alpha_{0,a,0,b}]$ & $[\alpha_{0,c+(-1)^da,0,b+d}]$ & $[\alpha_{1,c+(-1)^da,0,b+d}]$ \\
            $[\alpha_{1,a,0,b}]$ & $[\alpha_{1,c+(-1)^da,0,b+d}]$ & $[\alpha_{0,c+(-1)^da,0,b+d}]$
        \end{tabular}
    \end{equation*}
    where $a,c \in \Z$ and $b,d \in \Z_2$.
\end{proof}

\chunk{\label{defG3_b0}}
Consider
\begin{align*}
    G_1 &= \Z \\
    G_2 &= G_1 \rtimes_{\overline{\varphi}_1} \Z
    \quad \text{with} \quad
    \varphi_1 = [-1] \\
    G_3 &= G_2 \rtimes_{\overline{\varphi}_2} \Z
    \quad \text{with} \quad
    \varphi_2 = \begin{bmatrix} -1 & 0 \\ 0 & 1 \end{bmatrix}.
\end{align*}
Note that $G_3$ is well-defined as $\varphi_2 = \beta_0$ from \Cref{automorphisms}, and that
\[
    \langle g_1, g_2, g_3 \mid g_2g_1 = g_1^{-1}g_2, g_3g_1 = g_1^{-1}g_3, g_3g_2 = g_2g_3 \rangle
\]
is a polycyclic presentation of $G_3$. We begin by defining the multiplicative structure of $G_3$.

\begin{lemma} \label[lemma]{b0_exponents}
    Let $G_3$ be the polycyclic group defined in \ref{defG3_b0}. Then the following hold:
    \begin{arabiclist}
        \item \label{b0:1} $g_i^ag_1^b = g_1^{b(-1)^a}g_i^a$ for $i = 2,3$
        \item \label{b0:2} $(g_1^ag_2^bg_3^c)^m =
        \begin{cases}
            g_1^{ma}g_2^{mb}g_3^{mc} &\text{for } b+c \text{ even} \\
            g_1^{\mu(m)a}g_2^{mb}g_3^{mc} &\text{for } b+c \text{ odd}.
        \end{cases}$
    \end{arabiclist}
\end{lemma}

Using the relations defined in \Cref{b0_exponents}, we go on to define $\Aut(G_3)$.

\begin{theorem} \label{b0_automorphisms}
    Let $G_3 = \langle g_1, g_2, g_3 \mid g_2g_1 = g_1^{-1}g_2, g_3g_1 = g_1^{-1}g_3, g_3g_2 = g_2g_3 \rangle$ be the $3$-step poly-$\Z$ group defined in \Cref{defG3_b0}. Set
    \begin{align*}
        \mathcal{A} &= \left\{
            A \left|
            A = 
            \begin{bmatrix}
                2k & 2m+1 \\
                2l+1 & 2n
            \end{bmatrix}
            \right.
            \in \mathrm{GL}(2, \Z) \text{ where } k,l,m,n \in \Z
            \right\}, \\
        \mathcal{B} &= \left\{
            B \left|
            B =
            \begin{bmatrix}
                2k+1 & 2m \\
                2l & 2n+1
            \end{bmatrix}
            \right.
            \in \mathrm{GL}(2, \Z) \text{ where } k,l,m,n \in \Z
        \right\},
    \end{align*}
    and for $a \in \Z$ and $M \in \mathcal{A}\cup \mathcal{B}$, set
    \begin{align*}
        \alpha_{a, M} &=
        \left[
            \begin{array}{cc}
                1 & \begin{matrix} a & a \end{matrix} \\
                \begin{matrix} 0 \\ 0 \end{matrix} & M
            \end{array}
        \right],
        &\beta_{a, M} &=
        \left[
            \begin{array}{cc}
                -1 & \begin{matrix} a & a \end{matrix} \\
                \begin{matrix} 0 \\ 0 \end{matrix} & M
            \end{array}
        \right].
    \end{align*}
    Then
    \begin{align*}
        \Aut(G_3) &= \{\alpha_{a,A}, \alpha_{a,B}, \beta_{a,A}, \beta_{a,B} \mid a \in \Z, A \in \mathcal{A}, B \in \mathcal{B}\} \\
        \Inn(G_3) &= \{\alpha_{2a, I_2}, \beta_{2a, I_2} \mid a \in \Z\} \\
        \Out(G_3) &= \{[\alpha_{0,A}], [\alpha_{1,A}], [\alpha_{0,B}], [\alpha_{1,B}] \mid A \in \mathcal{A}, B \in \mathcal{B}\}.
    \end{align*}
    Moreover, the composition table of $\Out(G_3)$ is given by
    \begin{equation*}
        \begin{tabular}{c | c c}
            $\circ$ & $[\alpha_{0,M'}]$ & $[\alpha_{1,M'}]$  \\
            \cline{1-3}
            $[\alpha_{0,M}]$ & $[\alpha_{0,MM'}]$ & $[\alpha_{1,MM'}]$ \\
            $[\alpha_{1,M}]$ & $[\alpha_{1,MM'}]$ & $[\alpha_{0,MM'}]$
        \end{tabular}
    \end{equation*}
    for $M, M' \in \mathcal{A}\cup\mathcal{B}$.
\end{theorem}

\begin{proof}
    Let
    \[
        \psi =
        \begin{bmatrix}
            a & d & g \\
            b & e & h \\
            c & f & i
        \end{bmatrix}
        \in \Aut(G_3)
        \quad \text{ with } \quad
        \psi^{-1} =
        \begin{bmatrix}
            a' & d' & g' \\
            b' & e' & h' \\
            c' & f' & i'
        \end{bmatrix}.
    \]
    
    As the relations of $G_3$ must be preserved, we obtain $b, b', c, c' = 0$ and $e+f, e'+f', h+i, h'+i'$ are odd. From the identities $g_1 = \psi(\psi^{-1}(g_1))$, $g_2 = \psi(\psi^{-1}(g_2))$, and $g_3 = \psi(\psi^{-1}(g_3))$, we obtain $a = a' = \pm 1$, and we know that $e, e', i, i'$ must have the same parity, which is the opposite of that of $f, f', h, h'$. It also follows that if $a=1$ then $d'=-d$, while if $a=-1$ then $d'=d$. Thus, letting
    \[
        A =
        \begin{bmatrix}
            2k & 2m+1 \\
            2l+1 & 2n
        \end{bmatrix},
        B =
        \begin{bmatrix}
            2k+1 & 2m \\
            2l & 2n+1
        \end{bmatrix}
    \]
    such that $k,l,m,n \in \Z$ with $A,B \in \mathrm{GL}(2, \Z)$, we have that $\psi \in \{ \alpha_{d,A}, \alpha_{d,B}, \beta_{d,A}, \beta_{d,B}\}$.
    
    Further, it is easy to check that $\alpha_{a,A}, \alpha_{a,B}, \beta_{a,A}, \beta_{a,B} \in \Aut(G_3)$ for all $a \in \Z$ and $A,B$ as defined above, and that
    \[
        \alpha_{a,M}^{-1} = \alpha_{-a,M^{-1}},
        \quad
        \beta_{a,M}^{-1} = \beta_{a,M^{-1}}
    \]
    for any $M \in \mathrm{GL}(2,\Z)$ such that $M$ is of either form $A$ or $B$ as defined above.
    
    To identify $\Inn(G_3)$, let $\iota_h \in \Inn(G_3)$ where $h = g_1^ag_2^bg_3^c$ for some $a, b, c \in \Z$. It follows that
    \begin{align*}
        g_1^ag_2^bg_3^cg_1(g_1^ag_2^bg_3^c)^{-1}
        &= g_1^{(-1)^{b+c}} \\
        g_1^ag_2^bg_3^cg_2(g_1^ag_2^bg_3^c)^{-1}
        &= g_1^{2a}g_2 \\
        g_1^ag_2^bg_3^cg_3(g_1^ag_2^bg_3^c)^{-1}
        &= g_1^{2a}g_3.
    \end{align*}
    Therefore, $\Inn(G_3) = \{\alpha_{2a,I_2}, \beta_{2a,I_2} \mid a \in \Z\}$ and it can be verified that for any $a, b \in \Z$ and any $M \in \{A,B\}$ as defined above,
    \begin{equation*}
        \begin{tabular}{c|c c}
            $\circ$ & $\alpha_{b,M}$ & $\beta_{b,M}$ \\
            \cline{1-3}
            $\alpha_{a,I_2}$ & $\alpha_{2a+b,M}$ & $\beta_{2a+b,M}$ \\
            $\beta_{a,I_2}$ & $\beta_{2a-b,M}$ & $\alpha_{2a-b,M}$
        \end{tabular},
    \end{equation*}
    and thus $\Out(G_3) = \langle [\alpha_{0,A}], [\alpha_{1,A}], [\alpha_{0,B}], [\alpha_{1,B}]\rangle$ with composition structure
    \begin{equation*}
        \begin{tabular}{c | c c}
            $\circ$ & $[\alpha_{0,M'}]$ & $[\alpha_{1,M'}]$  \\
            \cline{1-3}
            $[\alpha_{0,M}]$ & $[\alpha_{0,MM'}]$ & $[\alpha_{1,MM'}]$ \\
            $[\alpha_{1,M}]$ & $[\alpha_{1,MM'}]$ & $[\alpha_{0,MM'}]$
        \end{tabular}
    \end{equation*}
    for $M, M' \in \{A, B\}$.
\end{proof}

\section{Acknowledgements}
I would like to thank Dr. Oana Veliche both for working with me as my thesis advisor and for continuing to advise me while I further developed this project. I would also like to thank Professors Anthony Iarrobino, Gordana Todorov, Bettina Eick, Alexandru Suciu, and Martin Kreuzer for their helpful suggestions and comments.

\end{document}